\documentclass[12pt,reqno,sumlimits]{amsart}

 \usepackage{amssymb,amscd,amsmath,epsfig,url,mathtools}
\usepackage{graphicx,type1cm,eso-pic}
\usepackage{hyperref}
\newtheorem{thm}{Theorem}[section]

\newtheorem{lem}[thm]{Lemma}
\newtheorem{prop}[thm]{Proposition}
\newtheorem{defn}{Definition}[section]
\newtheorem{rem}{Remark}[section]

\newcommand{\R}{{\mathbb R}}

\newcommand{\calE}{{\mathcal E}}

\renewcommand{\to}{\longrightarrow}

\newsavebox{\savepar}

\numberwithin{equation}{section}

\newcounter{labelflag} 
\newcommand{\labelon}{\setcounter{labelflag}{1}}
\newcommand{\Label}[1]{
                       \ifnum\thelabelflag=1
                          \ifmmode
                             \makebox[0in][l]{\qquad\fbox{\rm#1}}
                          \else
                             \marginpar{\vspace{0.7\baselineskip}
                                        \hspace{-1.1\textwidth}
                                        \fbox{\rm#1}}
                          \fi
                       \fi
                       \label{#1}
                      }
\labelon

\newcommand{\diffM}{{\rm Diff}(M)}
\newcommand{\diffzM}{{\rm Diff}_0(M)}
\newcommand{\dvol}{{\rm dvol}}
\newcommand{\np}{\nu_\phi}

\begin{document}

\newpage\null\vskip-4em
  \noindent\scriptsize{}

  \textsc{
MSC Classification: 58D15, 68T99. Keywords: Manifold Learning, Embedding Spaces, Discretized Gradient Flow
} 
\vskip 0.5 in

\normalsize

\title{Discretized Gradient Flow for Manifold Learning}

\author{ Dara Gold } \email{daracaseygold@gmail.com}
       \address{Jenner \& Block LLP (contractor)       } 
     
     \author{Steven Rosenberg} \email{sr@math.bu.edu}
       \address{Department of Mathematics and Statistics\\
       Boston University \\
       Boston, Ma 02215, USA}

\begin{abstract}
Gradient descent, or negative gradient flow, is a standard technique in optimization 
to find minima of functions. 
Many implementations of gradient descent rely on  discretized versions, {\it i.e.}, moving in the gradient direction for a set step size, recomputing the gradient, and continuing.
In this paper, we present an approach to manifold learning where gradient descent takes place in the infinite dimensional space 
 $\calE = {\rm Emb}(M,\mathbb{R}^N)$ of  smooth embeddings $\phi$ of a manifold $M$ into $\mathbb{R}^N$.
Implementing a discretized version of gradient descent for $P:\calE\to \R$, a penalty function that scores an embedding $\phi \in \calE$, requires estimating how far we can move in a fixed direction -- the direction of one gradient step -- before leaving the space of smooth embeddings. Our main result is to give an explicit lower bound for this step length in terms of the Riemannian geometry of $\phi(M)$. In particular, we consider the case when the gradient of $P$ is pointwise normal to the embedded manifold $\phi(M)$. We prove this case arises when $P$ is invariant under diffeomorphisms of $M$, a natural condition in manifold learning.
\end{abstract}

\maketitle

\section{Introduction}

A common approach in data analysis and machine learning is manifold learning, {\it i.e.}, determining how to approximate a finite set of $\{y_i\}_{i=1}^{L}$ in Euclidean space $\R^N$ by a $k$-dimensional embedded, compact manifold $M$ for some 
$k\ll N$ \cite{BNS, CSK, DonohoGrimes2003, GFDH,  MF, RS2000}.  
(The definition of embedding is at the end of the Introduction.)
While classic approaches to non-linear manifold learning include Isometric Mapping (IsoMap), Local-Linear Embeddings (LLE), and Laplacian and Hessian eigenmaps, there is a growing body of work that uses gradient descent of functionals to find manifold representations of high dimensional data.
This mathematical setup involves the space of smooth embeddings $\calE = {\rm Emb}(M,\mathbb{R}^N)$ considered as an open subset of the infinite dimensional vector space of all maps from $M$ to $\R^N$ with the Banach space topology coming from a high Sobolev norm or the $C^\infty$ Fr\'echet space topology. 
We also have a $C^1$ penalty function
$P : \calE \rightarrow \mathbb{R} $ which typically contains a data fitting term and a regularization term, as explained below.  (In keeping with the literature, we assume that $k$ and the diffeomorphism type of $M$ are given.)  In theory, finding a global minimum of $P$ via the negative gradient flow of $P$ on $\calE$ gives
an optimal embedding, or one that ``best fits" the training set $\{y_i\}$.  
The main result of this paper (Theorem \ref{thm:three})  gives precise bounds on the practical implementation issue 
of determining how far one can flow in a fixed gradient direction and still have an embedded manifold.

To avoid overfitting -- or choosing a $\phi$ such that $\phi(M)$
fits the $\{y_i\}$ very closely but performs poorly on new data points --  a penalty function $P$ can  penalize $\phi(M)$ both for being too far from  $\{y_i\}$ and for ``twisting too much" to fit the data. Thus a
  typical penalty function $P = P_1 + P_2:\calE\to\R$ 
  contains two terms:
(i) a data fitting term $P_1(\phi) = \sum_{i=1}^r d^2(\phi(M), y_i),$ where 
$d(\phi(M), y_i)$ is the Euclidean distance from $y_i$ to the closest point in $\phi(M)$;
 (ii)  
 a regularization term $P_2$ designed to prevent overfitting, e.g. 
$P_2(\phi) = \Vert \phi\Vert_s,$ the $s$-Sobolev norm of $\phi$. (For overviews of this standard approach, see \cite{Belkin}, \cite{Smola2001}.)    Gradient descent, {\it i.e.}, moving in the direction of $-\nabla P$ in $\calE$, can find a local or global minimum of $P$, or an optimal manifold to fit $\{y_i\}$.

While there are theoretical challenges with this setup, we focus on an implementation problem in this paper. In theory, to find a negative gradient flow line on  $\calE$, we need to know the gradient of $P$ at each point of $\calE.$ This is generally intractable for computer calculations.  Instead,  the gradient flow is often discretized: we move in the negative gradient direction from an initial point $\phi_0$ for a fixed step size to a new point $\phi_1$, stop and recompute the gradient at $\phi_1$, then iterate until the gradient is smaller than a specified amount. Since a gradient vector in the tangent space of $\calE$ corresponds to a vector field along $\phi(M)$, 
we  need to estimate 
a lower bound $t^* =t^*(\phi)$ 
for how far we can move from a fixed embedding $\phi$ in the negative gradient direction $-\nabla P_\phi$ and still remain in the space of embeddings. 
In summary, we avoid the usual problem that forward geometric flows tend to develop singularities by first discretizing the flow, and then estimating how big a gradient step avoids singularities.

This practical issue is the main focus of this paper. 
In the main result, Theorem \ref{thm:three}, we provide such a lower bound $t^*$, which in effect measures how well discretized flow can approximate the smooth flow. Here $t^*$
depends on  the  local and global extrinsic geometry of $\phi(M)$.

We emphasize that our approach to manifold learning directly tackles the  infinite dimensional nature of this optimization problem via gradient flow and without making any simplifying choices that reduce the problem to finite dimensions. Typical choices in the literature are parametric methods, which fix
a finite dimensional parameter space of embeddings, and RKHS methods, which  reduce the optimization to a finite dimensional problem via the Representation Theorem, but only after making a choice of kernel function. 
In contrast, our approach only assumes that $M$ is compact, possibly with boundary, and so must contend with infinite dimensional analytic issues. Since we are given a finite set of training data, the compactness assumption is reasonable.

We briefly discuss the issues with directly working with the smooth gradient flow on $\calE$. It may be difficult to prove that $P$ is differentiable for typical 
data terms 
which measure the minimum distance from a data point to the embedded manifold 
\cite{BR}. Even if $P$ is differentiable, in this infinite dimensional case it is not clear that a gradient flow line $\gamma(t)$ stays in $\calE$ or converges as $t\to\infty$ to a critical point of $P$, as  $\calE$  is an open dense set in the space of smooth maps from $M$ to $\R^N.$ 
Even if we can prove convergence, since neither $P$ nor $\calE$ is in general convex, a critical point 
need not be a global minimum, and a second derivative test for local minima may be difficult to develop and implement. Perhaps most fundamentally, even the short time existence for the gradient flow may be difficult to establish, particularly if we use the most natural $C^\infty$ topology on $\calE.$  These problems are well known in differential geometry, {\it e.g.,} in 
the study of minimal submanifolds.  
In contrast, discretized gradient flow is both a tool for theoretical results on gradient flow
\cite[Ch.~11]{AGS} and for  computer implementations based on discretized, usually linearized, versions of gradient flow \cite{DR}, although 
there may again be convergence issues \cite{Cooper}.

As an overview of the paper, in \S2, we give a short overview of manifold learning with references to the literature.  \S3 gives an outline of the proof of Theorem
\ref{thm:three}.  In \S4, we argue 
that the entire penalty term should be 
invariant under the diffeomorphism group of $M$, just like 
the data fitting term (i).
In particular, regularization terms built from geometric quantities like the volume or total mean curvature of $M$ have this invariance, while more familiar regularization terms like a Sobolev norm of the embedding do not. We prove in Theorem \ref{thm:one}  that for a diffeomorphism invariant penalty function $P$,  
the gradient vector field $\nabla P_{\phi}$ is guaranteed to be pointwise normal to $\phi(M).$  
\S5 gives the proof of Theorem \ref{thm:three}.  \S6 is a discussion of potential extensions of this work.  Appendix A contains a proof of a quantitative implicit function theorem used in \S5.
\medskip

We recall the technical definition of an embedding of a manifold $M$ into a manifold $W$ ($W= \R^N$ for us).\\

 \begin{defn}
 A smooth map $f:M\to W$ between smooth manifolds is an {\rm immersion} if the differential 
$df_x:T_xM\to T_{f(x)}W$ is injective for all $x\in M.$ An immersion is an {\rm embedding} if $f$ is a homeomorphism from $M$ to $f(M)$ in the induced topology, {\it i.e.,} a set  $V\subset f(M)$ is open iff $V = U\cap f(M)$ for an open set $U\subset W.$
\end{defn}

Since $M$ is compact in this paper, the unwieldy topological condition for an embedding 
simplifies.
\begin{prop}\label{prop:1.1} \cite[Prop.~4.22]{Lee}  If $M$ is compact, a smooth immersion $f:M\to W$ is an embedding.
\end{prop}

\section{Related work} 
Manifold learning is an approach to dimensionality reduction, the attempt to replace high dimensional data in $\R^N, N\gg 0$, by a low dimensional subset. 
Standard techniques in manifold learning, such as Locally Linear Embedding (LLE), IsoMap \cite{Tennenbaum}, Laplacian Eigenmaps \cite{Belkin},  and Hessian Eigenmaps \cite{DonohoGrimes2003}, involve  algorithms that reduce to (often nontrivial) minimization problems in finite dimensions.  In theory, these minimization problems can be solved by Lagrangian multipliers, so gradient descent is not a built in feature of these approaches. (We note that our discretization method is somewhat the reverse of the successful manifold approximation approach of Laplacian eigenmaps, where a discrete set of data in $\R^N$ that apparently lies close to a submanifold is parametrized by a subset of $\R^k$
 through eigenvectors of a graph Laplacian; this parametrization is our $\phi^{-1}.$)

In contrast, our approach is inherently infinite dimensional and relies on gradient flow, as explained in the Introduction.
The use of gradient flow for functionals on infinite dimensional manifolds of maps has a large literature  in machine learning, where this comes under the general heading of nonparametric methods.  (In the parametric approach, one restricts attention to a finite dimensional submanifold  depending on a finite dimensional family of parameters.)
Osher and Sethian introduced the Level Set Method \cite{OS}, which has been applied to machine learning by using gradient desecent on energy functionals (which act on the space of level set functions) to find optimal data-classification boundaries. 
Viewing the decision boundary this way avoids typical problems that arise with cusps and discontinuities in a flow whose speed is curvature dependent.  This work has been extended in many directions, including computer vision and image analysis, fluid mechanics, and classification problems  \cite{Mumford1989, Sethian1999, VW, Ye}. In supervised learning, \cite{BRSW} finds optimal statistical labeling functions by using gradient descent of penalty functionals that include both a data term $P_1$ as above 
and a geometric regularization term $P_2$.
(It should be noted that this paper has to resort to parametric methods to implement the discretized gradient flow algorithm.)  There are intriguing connections between regularization methods and classical physical equations in Lin {\em et al.}
 \cite{Lin2015}. 

Although applied here to manifold learning, the appearance of gradient flow in infinite dimensions of course has its roots in differential geometry.
In minimal submanifold theory, 
the penalty function is the purely geometric volume of the embedded manifold, and the gradient flow is the mean curvature flow.  A sampling of results is in \cite{Ger, Ham2, HS, RT, Xiao}.  Similar to our approach, Mayer \cite{May} uses a discretized approximation to the gradient flow, which more closely mimics implementation processes.  
It is worth noting that historically, 
 the modern study of gradient flow in differential geometry was initiated by Morse \cite{Morse} in the 1930s on the infinite dimensional space of paths on a Riemannian manifold, which was then adapted by Milnor \cite{Mil} to develop Morse theory on finite dimensional manifolds.  In turn, Morse theory has undergone widespread development through Floer theory and its many variants in the past 25 years (see {\it e.g.,} \cite{AD}).

As described in the Introduction, our penalty functions evaluate manifold embeddings with  diffeomorphism invariant data and geometric regularization terms.  The use of such terms is a developing area at the interface of differential geometry and machine learning. \cite{VW} uses the surface area of a decision boundary as the regularization term $P_2(\phi)$, while \cite{BRSW} uses the area of the manifold itself for $P_2(\phi)$, and \cite{Bergmann} uses a discrete version of the total mean curvature of a surface with applications to tomography. This last article contains many references to generalizations of optimization methods to a fixed finite dimensional Riemannian manifold, 
while our interest is in the infinite dimensional space of embeddings.
  Finally, the strongest connection to date between manifold learning and differential geometry is in the work of Fefferman {\it et al.} \cite{  feffermanReconstructionII, feffermanReconstructionI,
fefferman2019, fefferman2020, fefferman2016} on the ``manifold hypothesis." 

Although using gradient descent for manifold learning has widespread applications to machine learning, discretizing the flow - which is needed for most implementations - has many unstudied challenges. \cite{BRSW} for example, which is most closely related to our paper, uses a fixed step size in their gradient flow implementation. Our paper is the first to address the maximum step size that ensures a manifold remains an embedding when finding a low dimensional representation of training data.

\section{Proof Outline for the Discretized Gradient Flow Estimate}
Because of the computational detail in 
\S\ref{sec:flows}, we give an overview of the proof structure and the locations of key results.
\subsection{General Overview}
In Theorem~\ref{thm:one} in \S\ref{sec:condition}, we give a natural condition on the penalty function $P:\calE\to\R$ under which $\nabla P$ is pointwise normal  to an embedding $\phi(M)$. 
   Throughout the paper, we assume 
   that $P$ satisfies this condition.

 Given a pointwise normal vector field $ u$ along $\phi(M)$  with the length of each vector in $u$ at most one,
    \S\ref{sec:flows}, which has our main results, gives a lower bound
    for $t^*$ such that 
    $$\phi_t(m) = \phi(m) + tu_m$$ remains an embedding for all
     $|t| < t^*$.\footnote{  The Euler class  of the normal bundle $e\in H^{N-{\rm dim}(M)}(M)$ is the obstruction to the global existence of a unit normal vector field. Since $e$ may be nonzero,  
we must refer to vector fields whose elements have 
length at most one.  If $N > 2{\rm dim}(M)$, the obstruction vanishes.}  In particular, this applies to $u = k_\phi^{-1}\cdot\nabla P_\phi$, where $k_\phi = \max_{x\in M} \Vert \nabla P_{\phi(x)}\Vert$. Since $M$ is compact, it suffices to prove that  $\phi_t(M)$ is an injective immersion. 

   \subsection{Note on Computation of Key Values}
    Proposition \ref{prop:two} gives a condition under which $\phi_t$ is an immersion, and Theorem \ref{thm:two} defines the bound $t^*$ in which $\phi_t$ is injective. Finally, 
    together with our assumption that $M$ is compact, our main result
    Theorem \ref{thm:three} concludes the mapping is an embedding.
     In the proof of Theorem \ref{thm:two}, $t^*$ is initially a function of the quantities $\epsilon, \delta_H, \delta, K$. $\epsilon$ is defined in \S5.1(1), and
$K$ is explicitly defined in \S5.1(7)  as the maximal principal eigenvalue of $\phi(M).$
In Lemma 4, $\epsilon$ is computed as a function of $\delta$ and $K$, so $t^* = t^*(K, \delta, \delta_H).$  The dependence of $t^*$ on $\delta_H$ is eliminated after (\ref{tstar}), so finally $t^* = t^*(K, \delta).$

    The computation of $\delta$ is significantly more involved.   
    The characterizing property of  $\delta$ is in \S5.1(8). $\delta$ is  defined in (\ref{delta}) as the minimum of a quantity $\delta(q_0, v_0)$, 
    where $(q_0, v_0)$ is in the normal bundle  of $\phi(M)$.  In turn,  $\delta(q_0, v_0)$
    is computed in the proof of Proposition~\ref{prop:three} in three steps, each of which builds on the prior: $\delta^0(q_0, v_0)$ is defined in (\ref{do}), $\delta^1(q_0,v_0)$ is defined by (\ref{d1}), and $\delta^2(q_0,v_0)$ is defined in (\ref{d3}). Finally $\delta(q_0, v_0)$ is defined in (\ref{delta est}) in terms of $\delta^0(q_0, v_0), \delta^2(q_0, v_0)$.
    These steps are recapped in Remark \ref{rem:one}.

\section{A Condition for Normal Gradient Vector Fields}\label{sec:condition}

As outlined in the introduction, manifold learning involves searching for an embedding $\phi:M\to \R^N$ with $y_i\in {\rm Im}(\phi)$ for training data $\{y_i\}.$  
Of course,  $y_i\in {\rm Im}(\phi)$ iff $y_i\in {\rm Im}(\phi\circ g)$, where $g\in {\rm Diff}(M)$ is a diffeomorphism of $M$.
 Thus the penalty term $P_1:\calE\to\R$
 which measures goodness of fit should
not distinguish between $\phi$ and $\phi\circ g$, {\it i.e.,} this 
penalty term   
 must be invariant under the action of   ${\rm Diff}(M)$: $P_1(\phi) = P_1(\phi\circ g)$.
The data penalty term $P_1(\phi)= \sum_{i=1}^r d^2(\phi(M),y_i)$ in the introduction is clearly 
diffeomorphism-invariant.
(Since the quotient space $\calE/{\rm Diff}(M)$ may have a non-Hausdorff topology, we consider diffeomorphism-invariant penalty functions on $\calE$, rather than penalty functions on the quotient space.)  
These types of invariant functionals are familiar in  gauge theory, where functionals
are invariant under gauge group actions, and in Gromov-Witten theory, where maps are defined only up to holomorphic automorphisms.

Similarly, we can replace the non-diffeomorphism invariant regularization term $\Vert\phi\Vert_s$, which is computed in a choice of local coordinates, by e.g. $P_2'(\phi) = {\rm vol}(\phi(M))$, which measures a combination of the first derivatives of $\phi = (\phi^1,\ldots,\phi^N)$, or by\\ 
 $
P_2'(\phi)=\int_M \left[\sum_{j=1}^N (({\rm Id} + \Delta)^s\phi^j)\cdot\phi^j\right]^{1/2} \dvol_M$, which is equivalent to the $s$-Sobolev norm by the basic elliptic estimate.  As a simple example, for $\calE = {\rm Emb}(S^2, \R^3)$, $P_1'(\phi) = d^2(\phi(S^2),\vec 0)$, $P_2' = {\rm vol}(\phi(S^2))$, and 
for the standard unit sphere as the initial embedding $\phi_0(S^2)$,
gradient flow for $P' = P'_1+P'_2$ shrinks the unit sphere to the origin in infinite time.

In this section, we prove that such penalty functions have  gradients that are pointwise normal vector fields to $M$, and apply this result to $\calE$.  
We first review a known result about the gradient function on a finite dimensional manifold with a group action.  Recall that for a $C^1$ function $P:Z\to \R$ on an oriented Riemannian manifold $(Z,h)$, the gradient vector field $\nabla P$ is characterized by
$$dP_m(v) = \langle \nabla P, v\rangle_{h(m)},$$
for all $m\in Z, v\in T_mZ.$ Here $dP_m:T_mZ \to \R$, 
the differential of $P$ at $m$,  is independent of the Riemannian metric.

\begin{lem} \label{lem:one} Let $G$ be a connected Lie group acting via isometries on a Riemannian manifold $Z$.  A function $P:Z\to \R$ is $G$-invariant ($P(g\cdot m) = P(m)$ for all $m\in Z, g\in G$) iff $\nabla P_m$ is perpendicular to the orbit $\mathcal O_m = \{g\cdot m:g\in G\}$ for all $m\in Z.$
\end{lem}

Strictly speaking, we mean $\nabla P(m) \perp_{h(m)} T_m\mathcal O_m.$

\begin{proof}  If $P$ is $G$-invariant, then $\mathcal O_m$ is contained in a level set of $P$.  The gradient is always perpendicular to a level set:  for $X\in T_m\mathcal O$, take a curve $\gamma(t)\in \mathcal O_m$ with 
$\dot\gamma(0) = X$,  and compute
$$0 = (d/dt)|_{t=0} P(\gamma(t)) = dP_m(X) = \langle \nabla P_m, X\rangle.$$

\smallskip
Conversely, assume that $\nabla P_m\perp T\mathcal O_m$ for all $m$.  Take a smooth path $\eta(t), t\in [0,1],$ from $e\in G$ to a fixed $g\in G$, and for a fixed $m\in Z$ define $\gamma(t) = \eta(t)\cdot m.$ Then 
$$0 = \langle \nabla P_{\gamma(t)}, \dot\gamma(t) \rangle =  dP_{\gamma(t)}(\dot\gamma(t)),$$
so $P$ is constant along $\gamma(t).$  
In particular, $P(m) = P(\gamma(0)) = P(\gamma(1)) =P(g\cdot m)$.
\end{proof}
${}$\smallskip

We want to apply this result with $Z, G$ given by $\calE, \diffM$, respectively.  (Since $\diffM$ need not be connected, we have to restrict to $\diffzM$, the connected component of the identity diffeomorphism.) The smooth structure on mapping spaces is well known (see {\it e.g.,} \cite{Eells}).  Rather than go through the technicalities of the Lie group structure on $\diffM$ \cite{Omori}, we give a direct proof.

The tangent space $T_\phi\calE$ at an embedding $\phi$ is given by the infinitesimal variation of a family of embeddings $\phi(t)$, which for fixed $m\in M$ is given by $(d/dt)|_{t=0} \phi_t(m)\in T_{\phi(m)}\R^N\simeq \R^N.$ Thus elements $X$ of $T_\phi\calE$ are ``$\R^n$-valued vector fields along  $\phi(M)$," {\it i.e.,}  smooth functions
$X:M\to \R^N.$

For $\phi\in\calE$, $M$ has a Riemannian metric $g_\phi$ given by the $\phi$-pullback of the standard metric/dot product on $\R^N$ restricted to $\phi(M).$  Specifically, for $v,w\in T_mM$, 
$\langle v,w\rangle_m = d\phi(v)\cdot d\phi(w).$  Denote the associated volume form on $M$ by $\dvol_\phi.$
We take the $L^2$ inner product on $T_\phi\calE$ associated to 
the standard metric/dot product on $\R^N$ and $g_\phi$:
$$\langle X, Y\rangle_{\phi} = \int_M X_m\cdot Y_m\ \dvol_\phi(m).$$
Thus the gradient of $P:\calE\to \R$ is characterized by
$$ dP_\phi(X) = \langle \nabla P_\phi, X\rangle_\phi = 
 \int_{M}\nabla P_m \cdot X_m\ \dvol_\phi(m).$$
 
  $\diffM$ acts on $\phi\in\calE$ by $g\cdot \phi = \phi\circ g^{-1}.$ It is standard that $\diffM$ acts via isometries on $\calE$ with the $L^2$ metric.

In our setting, we can strengthen Lemma \ref{lem:one} to the  pointwise normal condition\\ 
$\nabla P_{\phi(m)}\cdot Q_{m} =0$ for all $Q_{m}\in T_{\phi(m)} \phi(M), m\in M$, for a $\diffM$-invariant $P$.

\begin{thm} \label{thm:one} For a $C^1$  function $P: \calE \to \mathbb{R}$, the gradient $\nabla P$ is pointwise normal to $T_{\phi(m)}\phi(M)$ for all $m\in M$ and for all 
$\phi\in \calE$ if and only if $P$ is invariant under diffeomorphisms in $\diffzM$, the path connected component of the identity in $\diffM.$
\end{thm} 

We note that this pointwise perpendicularity is measured in the usual dot product on $\R^N$, even though we have implicitly been using 
$\phi$-pullback metrics on $M$. In particular, the theoretical use of the pullback metric does not affect the practical implementation of discretized gradient flow.

\begin{proof}
Assume $P$ is $\diffzM$-invariant.  As in the Lemma, we conclude that $\nabla P_\phi \perp_{L^2} 
T_\phi \mathcal O_\phi.$

Take a family of diffeomorphisms $g_t$ of $M$ with $g_0 = {\rm Id}$ and with tangent vector 
$X = (d/dt)|_{t=0} g_t\in T_{\rm Id} \diffM.$  Then $\phi\circ g_t\in \mathcal O_\phi$, 
and the vector field $(d/dt)|_{t=0} \phi\circ g_t = d\phi (X)$ tangent to $\phi(M)$ is in $T_\phi \mathcal O_\phi$.  Conversely, any tangent vector field $V$ to $\phi(M)$ integrates to a family of diffeomorphisms in $\diffzM$, so we conclude that 
$V\in T_\phi\mathcal O_\phi$ and that (up to a choice of topology on $\diffM$)  $T_\phi \mathcal O_\phi$ is the space of tangent vector fields to $\phi(M).$

Fix $m_0 \in M$ and a vector $Q_{m_0} \in T_{\phi(m_0)}\phi(M)$. 
Choose a sequence $\epsilon_k\to 0$ and  smooth functions $f_{k} : \phi(M)\to \mathbb{R}$ such that $\int_{M}f_{ k}\ \dvol_\phi =1$, ${\rm supp} (f_{ k}) \subset B_{\epsilon_k}(\phi(m_0))  \cap \phi(M)$, with 
$B_{\epsilon_k}(\phi(m_0))$ the Euclidean ball of radius $\epsilon_k$ centered at $\phi(m_0)$. Extend $Q_{m_0}$ to a vector field $Q = Q_m$ on $\phi(M)$, and define the vector fields $Y_{k} $ on $\phi(M)$ by:
 $$
 Y_{k}(\phi(m)) = f_{k}(\phi(m))\cdot Q_m.
 $$
Then we have 
\begin{eqnarray*}
0&=&\lim_{\epsilon_k\to 0}
 \langle \nabla P _{\phi}, Y_{\epsilon_k} \rangle =
 \lim_{\epsilon_k\to 0}
 \langle \nabla P _{\phi}, f_{k}\cdot Q \rangle =
 \lim_{\epsilon_k\to 0}
 \int_{M} \nabla P _{\phi}(\phi(m))\cdot f_{k}(\phi(m)) Q_m\ \dvol_\phi \\  
&=&\nabla P_ {\phi}(\phi(m_0)) \cdot Q_{m_0}
\end{eqnarray*}
Therefore $ \nabla P_ {\phi}(\phi(m_0)) \perp Q_{m_0}$, and so 
$\nabla P_\phi(\phi(m_0))\perp T_{\phi(m_0)}\phi(M)$.

For the converse, assume
 that $\nabla P_{\phi}(\phi(m)) \perp T_{\phi(m)}\phi(M)$ for all $m \in M$. 
 Then $\nabla P \perp_{L^2} Q$ for all tangent vector fields $Q$ to $\phi(M)$, and so $\nabla P$ is perpendicular to the orbit of $\diffzM.$  As in Lemma \ref{lem:one}, we conclude that $P$  is $\diffzM$-invariant.
 \end{proof}

\section{Estimates for Flows in Normal Gradient Directions} \label{sec:flows}

Under the assumption that our penalty function is diffeomorphism invariant, to implement discretized gradient flow, by Theorem \ref{thm:one}
we have to know how far $\phi(M)$ can move in a fixed normal gradient direction while remaining in the space of embeddings. The next set of results gives an explicit estimate for the lower bound $t^*$ of this flow, with the main result in Theorem \ref{thm:three}.

Throughout the paper, we assume that $M$ is compact. By Prop.~\ref{prop:1.1},   $\phi:M\to \R^N$ is an embedding iff it is an injective immersion. 
Recall that $\phi$ is an immersion if its differential $d\phi$ is pointwise injective, which is the infinitesimal condition for  the map $\phi$ to be a local injection. 
Thus, there are two types of obstructions to a linearly deformed embedding $\phi_t$ of $\phi$ remaining an 
embedding: 
(1) a local obstruction, where distinct nearby points in $\phi(M)$ deform to the same point in 
$\phi_t(M)$;
(2)  a global obstruction, where points far from each other in the induced Riemannian metric on $\phi(M)$ deform to the same point in $\phi_t(M)$ because they are close in $\R^N.$  The local obstruction is controlled by the injectivity of the differential. Specifically, in Theorem \ref{thm:three}, we conclude that $t^*$ is ultimately a function of $ K$ and $\delta$, 
where $K$ is
is a bound on the principal curvature of $\phi$ and thus controls the local obstruction.
The global obstruction, which cannot be treated by infinitesimal means, is controlled in Theorem \ref{thm:three} by $\delta$, which is constructed by bounds in the Implicit Function Theorem.

\subsection{Notation and Definitions}

\begin{enumerate}

\item{} 
$\epsilon = \epsilon_{\phi}$ is chosen so that each $s$ in 
the $\epsilon$-neighborhood  $B_\epsilon(\phi(M))$ of $\phi(M)$ has a unique closest point $q= q(s)$ in $\phi(M)$. The existence of this neighborhood is guaranteed by the 
$\epsilon$-Neighborhood Theorem \cite[Thm.~6.24]{Lee}.   $B_\epsilon(\phi(M))$ is diffeomorphic to a neighborhood of the zero section of the normal bundle $\nu = \np$ of $\phi(M)$:
we have $s-q \in\nu_q =  \nu_{\phi,q}$, the
 fiber of $\np$ at $q$, and the map
$s\mapsto s-q$ is the diffeomorphism. 
A  lower bound for $\epsilon$ is given in Lemma \ref{lem:four} in terms of $\delta$ in (8) below; it will become explicit in Remark \ref{rem:one}.

\item{}  We  use two sets of coordinates on $\mathbb{R}^N$. Standard (global) coordinates are denoted $(x^1, \ldots, x^N)$. We  also represent points  $s \in B_\epsilon(\phi(M))$
 as $$s=(q^1, \ldots, q^k, v^1, \ldots, v^{N-k}) = (q, v),$$
 where the $q^i$ are local manifold coordinates and $v^j$ are local coordinates for the normal space. These are called normal coordinates. Thus  $q \in \phi(M)$ has $q = (q^1, \ldots, q^k, 0, \ldots, 0)$. Here $k = {\rm dim}(M).$ Note that normal coordinates are not well defined outside $B_\epsilon(M).$

\item{} A vector in $\np$ can be expressed either as $tv_q$, where $v_q$ is a unit length vector
 at $q$,
 or as $v^iw_{i,q}$, where $\{w_{i,q} \}$ is an orthonormal basis of $\nu_{\phi,q}$. There are $N-k$ $\{w_{i,q} \}$ vectors, each with $N$ Euclidean coordinates.
 
\item{} 
The endpoint map $E: \np \rightarrow \mathbb{R}^N$ is  $E(q,v) =q+v$. 
 It is given explicitly by:
$$
 	E( q^1, \ldots, q^k, v^1, \ldots, v^{N-k}) = (x^1(q)+ v^iw_{i,q}^1, \ldots, x^N(q)+ v^iw_{i,q}^{N}),
$$   
 where the domain is in normal coordinates and the range is in standard coordinates. 

 \begin{defn} \label{fp} \cite[\S6]{Mil} $e =q_e+v_e\in B_\epsilon(\phi(M))$ is a focal point  
 if the Jacobian of the $E$ map is not full rank at $(q_e,v_e)$.
 \end{defn}

\item{} The inclusion map $\phi(M) \rightarrow \mathbb{R}^N$ is  $ q = (q^1, \cdots, q^k) \mapsto (x^1(q),\cdots, x^N(q)) = x( q)$ in manifold to Euclidean coordinates, so 
 the first fundamental form is the matrix $(g_{ij}) = \big(\frac{\partial x}{\partial q^i} \cdot \frac{\partial x}{\partial q^j}\big) $, where $\cdot$ is the Euclidean dot product. The second fundamental form at the normal vector 
 $v\in \np$ is the matrix  ${\rm II}_{v} = \left(v \cdot 
 \frac{\partial^2  x}{\partial q^i\partial q^j}\right)$.

\item{} At a fixed  $ q\in\phi(M) $, we may choose manifold coordinates so that
 the first fundamental form is the identity matrix.  The principal curvatures of $ v$ at $ q$ are by definition  the eigenvalues $p_1, \ldots, p_k$ of ${\rm II}_{ v}$.  Here $p_i = p_i(q,  v).$

\item{} Let $K$ be the maximal
principal eigenvalue of $\phi(M)$.  Thus we take the maximum of the $p_i(v)$ over all unit vectors 
in $\np.$ 
\item{} \label{deltadef} $\delta$ is chosen such that normal lines of length $\epsilon$ 
based at different, close points of $\phi(M)$ do not intersect: for $d_{\mathbb{R}^N}(\phi(m_1), \phi(m_2)) < \delta$, 
$\phi(m_1) + t_1v_1 \neq \phi(m_2) + t_2v_2$ for unit normal vectors $v_i \in \nu_{\phi(m_i)}$,
$i = 1,2$, and $|t_1|, |t_2| < \epsilon,$  with $\epsilon$ defined in (1) above.
$\delta$ is precisely defined in (\ref{delta}), and estimated in Remark \ref{rem:one}. ($\delta$ is 
 the {\em reach} of 
$\phi(M)$, as  in {\it e.g.} \cite{fefferman2016}.)

\end{enumerate}

\begin{rem}\label{rem:zero}  {\rm
In the calculations below, estimates for $\epsilon, \delta, K$ are computed explicitly in terms of $\phi$, local coordinates on $M$, and local coordinates on $\nu_\phi$.  Specifically, a lower bound for $\epsilon$ 
in terms of $K$ and $\delta$ is given in Lemma 4.  $K$ of course depends on $\phi$, but is in fact independent of coordinates on $M$, as it is the maximum eigenvalue of any normal component of the trace of the second fundamental form. 
   The estimate of $\delta$ uses $\phi$,  local coordinates on $M$, and local coordinates on $\np$
in {\it e.g.}, the proof of Proposition \ref{prop:three}. It is reasonable to assume knowledge of coordinates on $M$, as a manifold is specified by a cover of charts.   In fact, local coordinates on $M$ and $\phi$ determine local coordinates on $\np.$\footnote{Take the standard basis $\{e_i\}$ of $\R^N$. For  
$I = (i_1,\cdots ,i_{N-k})$ with $1\leq i_1< \cdots < i_{N-k}\leq N,$ lexicographically ordered, set
 $e_I = (e_{i_1},\ldots, e_{i_{N-k}})$  
Let $U_I$ be the open set of $q\in \phi(M)$ such that $I$ is the smallest  multi-index such that the projection of 
$e_I$ into $\nu_{\phi,q}$ is a basis of $\nu_{\phi,q}$. Then $\np$ is trivial over $U_I$, and we can form a new, 
fixed cover of $M$ by taking $\{V_i\cap U_I\}.$  In particular, the local coordinates on $\np$ in (2) are not extra 
data, since the embedding $\phi$ determines which $q$ are in which $U_I$.}  Thus, in the end our estimates depend only on local coordinates on $M$ and on $\phi.$ 
See Remark \ref{rem:one} for more details.  }
\end{rem}

\subsection{Calculating the Flow Length to Remain an Embedding}

In this section, we compute $t^*$ such that for $t < t^*$ and $u$ a normal vector field along $\phi(M)$
with $|u_{\phi(m)}| \leq 1$,  the deformed manifold $\phi_t(M) = \{ \phi(m) + tu_{\phi(m)}:m\in M\}$ is an embedding. As above, it suffices to prove that each $\phi_t$ is an immersion.

We start by determining which normal deformations  $\phi_t(M)$ of $\phi(M)$ are still immersions. 

\begin{prop}\label{prop:two}
Let $u$ be a normal vector field of length at most one along $\phi(M) \subset \mathbb{R} ^{N}$, and let $\epsilon$ be defined in \S5.1(1). Then $\phi_t(M) = \{{\phi(m) + tu_{\phi(m)} : m \in M} \}$ is immersed in $\mathbb{R}^{N}$ for $|t| < \epsilon$.
\end{prop}

\begin{proof}
Because $\phi:M\to \R^N$ is an embedding,  it suffices to show that the map $F_t:\phi(M) \rightarrow \phi_t(M)$,  $F_t(q) = q +tu_q$, is an immersion. 
In normal coordinates, we have
$$
F_t(q^1, \ldots, q^k) = (q^1, \ldots, q^k, tu^1_q, \ldots, tu^{N-k}_q).
$$
The differential $DF_t$, written as an $N\times k$ matrix, is of the form
$$DF_t = \left(\begin{array}{c}\\ {\rm Id}_{k\times k}\\ \\ \hline\\
{}^\star\end{array}\right),$$
where 
$\star$ is some $(N-k)\times k$ matrix.
This has rank $k$, so $F_t$ is an immersion. We note $\epsilon$ is implicitly used as normal coordinates are only defined in $ B_\epsilon(\phi(M)).$ 
\end{proof}

 Thus $\phi_t$ is an embedding if it is injective.
 Theorem \ref{thm:two} proves injectivity for $|t|\leq t^*$, where $t^*$ is defined in the Theorem statement.
The proof of Theorem \ref{thm:two} follows after the proofs of Lemmas \ref{lem:two}-\ref{lem:four} and Proposition \ref{prop:three}.

\begin{thm}\label{thm:two}
Let $u$ be a  normal vector field of length at most one along $\phi(M) \subset \mathbb{R} ^{N}$ Let $t^* = \min\{K^{-1}, \delta/3\}$. Then $\phi_t: M \rightarrow \mathbb{R}^N$ given by $m \mapsto \phi(m) + t 
u_{\phi(m)}$ is injective for $|t| \leq t^*$.
\end{thm}

Here $\delta$ is given by \S5.1(8), and will be estimated explicitly after the proof of Proposition \ref{prop:three}.
\medskip

\noindent {\it Proof.}
 As in the previous proof, it suffices to show that $F_t:\phi(M) \to \phi_t(M)$  is injective.
We extend $F_t$ to a 
 map between open subsets of $\R^N$ by setting
$$H_t: B_{\epsilon -t}(\phi(M)) \to B_{\epsilon}(\phi(M)), \ \  H_t(b) = b + t u_{q(b)},$$
where $q(b)$ is the closest point in $\phi(M)$ to $b$.
Note that $H_t|_{\phi(M)} =F_t$ and that $H_t$ is defined only for $|t| <\epsilon.$

We now proceed with a series of Lemmas.
\begin{lem}\label{lem:two}  For each $q\in \phi(M)$, 
there exists a ball $B_{\delta^{q}_{H_t}}(q)$ of radius $\delta_{H_t}^{q}$ around $q$ on which $H_t$ is a diffeomorphism. 
 \end{lem}

 \begin{proof}
In normal coordinates, we have
 $$
H_t(b) = H_t (q^1, \ldots, q^k, v^1, \ldots, v^{N-k}) = (q^1, \ldots, q^k, v^1 + tu^1_{q(b)}, \ldots,v^{N-k} + tu^{N-k}_{q(b)}).
 $$
 For $q = (q,0) \in \phi(M)$, the differential of the $H_t$ map has the matrix
 $$DH_t(q) = 
 \left( \begin{array} {c|c}{\rm Id}_{k\times k}& \frac{\partial q^i}{\partial v^j}\\
 &\\
 \hline\\
 \frac{\partial (v^i+ tu^i_q) }{\partial q^j}& \frac{\partial (v^i+ tu^i_q)}{\partial v^{j}} 
  \end{array}\right)
 = \left( \begin{array} {c|c}{\rm Id}_{k\times k}& 0\\
 &\\
 \hline\\
 \frac{\partial (v^i+ tu^i_q) }{\partial q^j}& {\rm Id}_{(N-k)\times (N-k)}\end{array}\right).
$$

  \noindent 
   This matrix is invertible, so the Lemma follows from the inverse function theorem. 
   \end{proof}

   Let $\delta_{H_t} =\min_{q}\{ \delta_{H_t}^{q}\}$.
 Set
  \begin{equation}\label{H} \delta_{H} = \min\{\delta_{H_t}: |t| \leq .999\epsilon\}.
  \end{equation}
 From the proof of the inverse function theorem, we can choose $\delta^q_{H_t}>0$ to be continuous in $t$. We need $|t|< \epsilon$, and then  the further restriction $|t| \leq .999\epsilon$ ensures that $t$ lies in a compact subset of $\R$.  Thus  $\delta_{H_t}$ and  $\delta_H$ are positive. 
  Note that $\delta_H = \delta_H(u)$ depends on the choice of the normal vector field $u$.

  \begin{lem} \label{lem:three}    $H_t|_{\phi(M)}$ is injective for $|t| < t^* 
  \stackrel{\rm def}{=} \min\left\{\epsilon, \delta_{H}/3\right \}$.
 \end{lem}

 \begin{proof}
 Assume instead that there exist $x, y \in \phi(M)$ such that $x + t u_x = y + t u_y$ for $|t| < t^*$. By Lemma \ref{lem:two}, $d_{\mathbb{R}^N}(x,y)> \delta_{H_t}$.
Then
  \begin{eqnarray*}
    \delta_{H_t} &<& d_{\mathbb{R}^N}(x,y) =|x-y|
    = |x - (x + t u_x) + (x + t  u_x) -y| \\
     &\leq& |x - (x +  t  u_x)| +  |(y +  t  u_y) -y| 
     = | t  u_x| + | t  u_y| 
     \leq 2|t| < 2t^* \\
     &\leq& 2\delta_{H_t}/3, 
     \end{eqnarray*}
since $t^* < \delta_{H_t}/3.$  This is a contradiction. 
    \end{proof}
  
We now compute $\epsilon$ in \S5.1(1) in terms of $K$ in \S5.1(7) and $\delta$ in \S5.1(8).  As mentioned above, $K$ is computed locally on $\phi(M)$, while $\delta$ is computed globally using the Euclidean distance.

\begin{lem}\label{lem:four} Set $\epsilon = \min \left\{K^{-1}, \delta/3 \right\}$, where $K$ is given in \S5.1(7) and $\delta$ is given in \S5.1(8).  Then every point in $B_\epsilon(\phi(M))$ has a unique closest point in $\phi(M).$
\end{lem}

 \begin{proof} 
  By
\cite[Lem.~6.3]{Mil}, the focal points (Def.~\ref{fp}) of $\phi(M)$ along the normal line $l =  q + t v$ are precisely the points $q + p_i^{-1}v$, where 
the $p_i$ are the nonzero principal curvatures.
The proof of the $\epsilon$-Neighborhood Theorem in \cite[Thm.~6.24]{Lee} uses the invertibility of the endpoint map, so we must have
$\epsilon < K^{-1}.$
 
 Suppose there exists $b \in B_\epsilon(\phi(M))$ with closest points $x, y \in \phi(M)$. Then  $b= x+t  v_x = y+t'  v_y$ for unit normal vectors $v_x$ at $x$, $v_y$ at $y$, and 
 $|t|,|t'| < \epsilon.$  By definition of $\delta$, we have $d_{\mathbb{R}^N}(x,y) > \delta$.  As in the previous proof, we have
\begin{eqnarray*}
\delta < d_{\mathbb{R}^N}(x,y) &=&|x-y|
=|x-(x+t  v_x) + ( y+t'  v_y) -y| \\
      &\leq& |t| |  v_x| + |t'| |  v_y|
       < 2\epsilon \leq 2\delta/3, 
     \end{eqnarray*} 
  a contradiction.
  \end{proof}

     \medskip

 We can now define $\delta$ in (\ref{delta}) below, after which we explicitly estimate it in the proof of Proposition~\ref{prop:three}. The steps of the estimate are recapped in Remark \ref{rem:one}.
 We first restrict the endpoint map 
  $E: \nu_\phi \to \mathbb{R}^N$ 
to the compact set $W = \{v\in \nu_\phi:|v| \leq .999K^{-1}\}.$
  For fixed $q_0 \in \phi(M)$ and  $(q_0,v_0)\in \nu_{\phi,q_0}\cap W = W_{q_0}$,
   the proof of Lemma 
  \ref{lem:two} shows 
that $DE(q_0,v_0)$ is invertible. 
Therefore, there is a ball of radius $\delta(q_0,v_0)>0$ around $(q_0,v_0)$ on which $E$ is a diffeomorphism. 
  Set $\delta_{q_0} = \delta(q_0,0)$ and
  \begin{eqnarray*}
  A_{q_0}&=&\{ q \in \phi(M): d_{\mathbb{R}^N} \: (q, q_0) < \delta_{q_0} /2 \}.
  \end{eqnarray*}
   We claim that  $E$ is a diffeomorphism on the  the set $B_{q_0} \subset \nu_\phi $ given 
   by
$$
B_{q_0} = \{ (q,v): |v|< \delta_{q_0} /2, q \in A_{q_0} \}.$$
Indeed, for
 $(q_1, v_1)\in B_{q_0}$, we have
$$
 |(q_1, v_1) - (q_0,0)| 
 \leq |(q_1, v_1) -(q_1,0)| +|(q_1,0)- (q_0,0)| 
\leq  \delta_{q_0} + \delta_{q_0} /2
\leq \delta_{q_0}.
$$
 Thus for $(q_1, v_1), (q_2, v_2) \in B_{q_0}$ and $(q_1, v_1)\neq (q_2, v_2)$, we conclude $(q_1,0), (q_2,0) \in A_{q_0}$ and $E(q_1,v_1)  \neq  E (q_2,v_2)$.  Since $E$ is invertible on $B_{q_0},$ it 
 is a diffeomorphism onto its image.

  We set
  \begin{equation}\label{delta}
   \delta 
   = \frac{1}{2} \min\{\delta(q_0,v_0): (q_0,v_0)\in \nu_\phi, |v_0| \leq .999K^{-1}\}.
   \end{equation}
Since $M$ is compact and $|v_0|$ lies in a compact interval, $\delta$ is positive.
In other words, for  $q_1, q_2 \in \phi(M)$ with $d_{\mathbb{R}^N}(q_1,q_2) < \delta$, we have
 $q_1+v_1\neq q_2+v_2$
 for $|v_1|, |v_2|< \delta$ and $(q_1,v_1)\in\nu_{\phi,q_1}, (q_2,v_2)\in \nu_{\phi,q_2}.$

  For a fixed $(q_0,v_0)$, it remains to compute $\delta(q_0,v_0)$ explicitly, after which  $\delta$ in (\ref{delta}) is explicit.
  The computation of $\delta(q_0,v_0)$ uses a quantitative version \cite{CL} of the Implicit Function Theorem 
  given in the next Proposition. The proof is in the Appendix.

To set the notation, let the matrix norm $\Vert A\Vert$ be the sup norm of the absolute values of the entries.
For $G \in C^1(\mathbb{R}^{m+n}, \mathbb{R}^m)$, let $(s_0, y_0) \in \mathbb{R}^{m+n} \times \R^m$ satisfy $G(s_0, y_0) =0 $.
 For fixed $\delta > 0$, set $V_{\delta} = V_{\delta(s_0,y_0)}= \{(s,y) \in \mathbb{R}^{m+n}: |s-s_0| \leq \delta, |y-y_0| \leq \delta \}$.
We focus on the case $G(s,y) = E(s) - y$ for  $m=n$, the usual method to derive the Inverse Function Theorem from the Implicit Function Theorem.

\begin{prop} \label{prop:three}Assume that the $m\times m$ matrix $\partial_{s }G(s_0,y_0)$ of partial derivatives of $G$ in 
the $s$ directions is invertible. Choose $\delta^0 > 0$ such that 
\begin{equation}\label{one}\sup_{(s,y)\in V_{\delta^0}}\Vert {\rm Id} - [\partial_{s }
G(s_0,y_0)]^{-1}\partial_{s }G(s,y)\Vert \leq  1/2.
\end{equation}
 Set \\
{\rm (I)}\  \  \  $B_{\delta^0} = \sup_{(s,y)\in V_{\delta^0}}\Vert\partial_y  G(s,y)\Vert$,\\
{\rm (II)}\ \  \ $P= \Vert\partial_{s }G(s_0,y_0)^{-1}\Vert$,\\
{\rm (III)}\ \  $ \delta^1 = (2PB_{\delta^0})^{-1}\delta^0$.\\
Then for the case $n=m$ and
$G(s,y) =E(s) -y$, 
on the set 
$\{ (s,y) : \Vert s-s_0\Vert < \delta^0, \Vert y-y_0
\Vert < \delta^1, G(s,y) = 0 \}$, 
$E$ has a $C^1$ inverse: $E(s) = y$ iff $s= E^{-1}(y).$
 Equivalently, $E$ is a $C^1$ diffeomorphism on 
 \begin{equation}\label{oneA} E^{-1}(B_{\delta^1}(y_0)) \cap B_{\delta^0}(s_0).
 \end{equation}
\end{prop}

\medskip
To apply the Proposition, we set $n=m= N$ and $ G((q,v),y) = E(q,v) -y$, where $E$ is the endpoint map. 
We follow the Proposition's labeling in a series of steps: \\
\\
\textbf{ Criterion I:} Independent of the value of $\delta^0 = \delta^0((q_0, v_0), y_0)$, we have

\begin{eqnarray*}
B_{\delta^0} &=& \sup_{((q,v),y)\in V_{\delta^0}}||\partial_y G((q,v),y)||
= \sup_{((q,v),y)\in V_{\delta^0}}||\partial_y (E(q,v)-y)||\\
 &=&\sup_{((q,v),y)\in V_{\delta^0}} \Vert -{\rm Id}\Vert = 1.
\end{eqnarray*}

\noindent\textbf{Criterion II:} By \S5.1(4),(7), 
\[ \partial_{(q,v)}G((q_0,v_0),y_0)=DE(q_0,v_0)\]
is invertible for $|v|< K^{-1}$. 
In the notation of \S5.1(4), 
\begin{align}\label{quick2} \lefteqn{DE(q_0,v_0)=  }\nonumber\\
 & \left( \begin{array} {cccccc}
 \left(\frac{\partial x^1}{\partial q^1}+v^i\frac{\partial w_i^1}{\partial q^1}\right)|_{(q_0,v_0)} & \cdots 
 & \left(\frac{\partial x^1}{\partial q^k}+v^i\frac{\partial w_i^1}{\partial q^k}\right)|_{(q_0,v_0)}
 &  
w^1_{1,q_0}&\cdots& w_{N-k,q_0}^1\\
 \vdots & & \vdots&\vdots&&\vdots \\
 \left(\frac{\partial x^N}{\partial q^1}+v^i\frac{\partial w_i^N}{\partial q^1}\right)|_{(q_0,v_0)} & \cdots &  
 \left(\frac{\partial x^N}{\partial q^k}+v^i\frac{\partial w_i^N}{\partial q^k}\right)|_{(q_0,v_0)}&w^N_{1,q_0}&\cdots&
 w_{N-k,q_0}^N \end{array} \right) 
 \end{align}
${}$
By Cramer's rule, 
\begin{equation}\label{M}P 
=\Vert DE(q_0,v_0)^{-1}\Vert
 = (\det(DE (q_0,v_0)))^{-1} \Vert(DE(q_0,v_0)^*\Vert,
 \end{equation} 
 where $DE(q_0,v_0)^*_{(i,j)}$ is the usual minor of $DE(q_0,v_0)$ obtained by deleting the $i^{\rm th}$ row and 
 $j^{\rm th}$ column.  Since $\phi$ and the $w_i$ are given, we obtain an estimate for $P$.

  \noindent {\bf Criterion III:}    
      We
 now compute  $\delta^1 = \delta^1(q_0,v_0), \delta^0 = \delta^0(q_0,v_0)$ such that (\ref{one}) holds for $((q,v),y)$.
    Since 
   (\ref{one}) is independent of $y$ in our case,  
    we need $\delta^0(q_0,v_0)$ such that  
    \begin{equation}\label{two}|(q,v)|< \delta^0(q_0,v_0) \Rightarrow 
    \Vert {\rm Id} - [DE(q_0,v_0)]^{-1}DE(q,v)\Vert \leq  1/2 .
    \end{equation}
We consider a first order Taylor  expansion of $DE(q,v)$  around $s_0 = (q_0, v_0)$. (Note: The summed index $j$  below refers to coordinates in $\mathbb{R}^N$, not an exponent). For $s = (q,v)$, we have
\begin{eqnarray}\label{three}  DE(s) &=&  
DE(s_0)     + \left( \begin{array} {ccc}
      R^{(1,1)}_j (q,v)(s-s_o)^j & \cdots &
      R^{(1,N)}_j (q,v)(s-s_o)^j \\
      \vdots & & \vdots \\
       R^{(N,1)}_j (q,v)(s-s_o)^j & \cdots &
       R^{(N,N)}_j (q,v)(s-s_o)^j  \end{array} \right)\\
       &\stackrel{\rm def}{=} & DE(s_0) + (R^{(p,r)}_j (q,v)(s-s_o)^j).\nonumber
    \end{eqnarray}  
  As in Criterion II, set
  $f^r_p= \frac{\partial x^r(q)}{\partial q^p}+v^i\frac{\partial w_i^r(q)}{\partial q^p}$  for  all $1\leq p\leq N$, 
   $1\leq r \leq k$, and $f^r_p= w^r_{p,q}$ for
 $1 \leq p \leq N$, $k+1\leq r \leq N$.
 A uniform bound on the error term is given by Taylor's theorem with integral remainder:
 \begin{align*}\label{cs}  \left|  R^{(p,r)}_j   
 (q,v)
 (s-s_0)^j\right|
  &\leq 
\left\vert \int_0^1 (1-t) \partial_j f^r_p((1-t)(q_0,v_0) + t(q,v)) dt\right\vert\cdot\left\vert
(s-s_0)^j\right\vert \\
&\leq \max  \left\{\left|\partial_j f^r_p(q,v)\right| : 1 \leq j \leq N, |v| \leq .999 
  K^{-1}, q \in \phi(M) \right \}
   |s-s_0|\nonumber\\
  &\stackrel{\rm def}{=} L^{(p,r)}_j 
   |s-s_0|.\nonumber
\end{align*}
 Here $\partial_j$ differentiates in the $s$ variable.
Set
 \begin{equation}\label{G} {L} = \max_{j,p,r} \{{L}^{(p,r)}_j\}. 
 \end{equation} 
 
  Plugging (\ref{three}) into the right hand side of (\ref{two}) and canceling the identity matrix, the matrix norm in
  (\ref{two}) becomes
   \begin{eqnarray}\label{four}  \left \Vert [DE(q_0,v_0)]^{-1} (R^{(p,r)}_j (q,v)(s-s_0)^j)\right\Vert
&=& \max_{j,p,r}\left\vert ([DE(q_0,v_0)]^{-1})^{p}_{\ell}
(R^{(\ell,r)}_j (q,v)(s-s_0)^j)\right\vert \nonumber\\
&\leq& N \Vert [DE(q_0,v_0)]^{-1}\Vert \cdot {L} \cdot \delta^0(q_0,v_0),
\end{eqnarray}
where the $N$ comes from the sum over $\ell = 1,\ldots, N$.
        Setting  
        \begin{equation}\label{do} \delta^0(q_0,v_0) = \left[2N\Vert DE(q_0,v_0)^{-1}\Vert\cdot {L} \right]^{-1},
        \end{equation} 
  we conclude  that the estimate (\ref{two}) is satisfied.

 In summary, we now have
\begin{equation}\label{d1}
        \delta^1(q_0,v_0) = (2PB_{\delta^0(q_0,v_0)})^{-1}\delta^0(q_0,v_0) =(2P)^{-1}\delta^0(q_0,v_0) ,
 \end{equation}
by Criterion I.  Thus $\delta^1(q_0, v_0)$ is estimated by Criterion II and III. 

\medskip

        By Proposition \ref{prop:three},  $E$ is a diffeomorphism on $
        E^{-1}(B_{\delta^1(q_0,v_0)}(y_0)) \cap B_{\delta^0(q_0,v_0)}(q_0,v_0)$. To be  explicit, we want to find radius $\delta(q_0,v_0)$
             such that
        \begin{equation}\label{radius} B_{\delta(q_0,v_0)}(q_0,v_0) \subset  E^{-1}(B_{\delta^1(q_0,v_0)}(y_0)) \cap B_{\delta^0(q_0,v_0)}(q_0,v_0).
        \end{equation}
   
 We first find  $\delta^2(q_0,v_0)$ such that  
 $$|(q,v)-(q_0,v_0)|<\delta^2(q_0,v_0) \Rightarrow |E (q,v) - E(q_0,v_0)| = |E (q,v) - y_0| < \delta^1(q_0,v_0).$$
 In  other words, we want
\begin{equation}\label{seven}|(q,v)-(q_0,v_0)|<\delta^2(q_0,v_0) \Rightarrow
E(q,v) \in B_{\delta^1(q_0,v_0)}(y_0).
\end{equation}
As above, we compute $\delta^2(q_0,v_0)$ by a Taylor series expansion of $E$ around $(q_0,v_0)$:
        $$
        E(q,v) = E(q_0,v_0)+ \left( \sum\limits_j R^1_j(q,v)((q,v)-(q_0,v_0))^j, \ldots, \sum\limits_j R^N_j(q,v)((q,v)-(q_0,s_0))^j \right),
        $$
with 
\begin{eqnarray}\label{gp}
        |R^p_j(q,v)| &\leq& \max \left\{ \left|\partial_j (\phi^p+ v^iw_i^p)(q,v) \right| : 1 \leq j \leq N, |v| \leq .999 K^{-1}, q \in \phi(M) \right\} \nonumber\\
        &\stackrel{\rm def}{=}&  {S}^{p}. 
        \end{eqnarray}
For $s_0 = (q_0,v_0), s = (q,v)$, we have
         \begin{eqnarray*}
          |E(s) -E(s_0)|^2
          &=& \sum\limits_{p=1}^N \left(\sum\limits_jR^p_j(s)(s-s_0)^j\right)^2
 \leq \sum\limits_{p=1}^N \left(\sum\limits_j|R^p_j(s)|^2\right) |s-s_0|^2\\ 
&\leq& N\left(\sum\limits_{p=1}^N |{S}^p|^2 \right)|s-s_0|^2
           \leq \sum\limits_{p=1}^N \sum\limits_j|{S}^p \delta^2(q_0,v_0)|^2.
           \end{eqnarray*}
Therefore, for  
\begin{equation}\label{d3}\delta^2(q_0,v_0) = \delta^1(q_0,v_0)\left (N\sum\limits_{p=1}^N 
|{S}^p |^2 \right )^{-1/2},
\end{equation}
estimate (\ref{seven}) holds.  Finally, setting
\begin{equation}\label{delta est}\delta(q_0,v_0) = \min \{\delta^2(q_0,v_0), \delta^0(q_0,v_0)\}
\end{equation}
  accomplishes (\ref{radius}).
 \medskip

By Lemmas \ref{lem:three}, \ref{lem:four}, and using (\ref{delta}) to define $\delta$,  we know that Theorem \ref{thm:two} holds,  {\it i.e.,} $\phi_t$ is injective,   for 
\begin{equation}\label{tstar}t^*
 < \min\{K^{-1}, \delta_{H}/3, \delta/3\}.
 \end{equation}
 If we prove that $\delta_H> \delta$, then we get injectivity of $\phi_t$ for $t^* <\min\{K^{-1}, \delta/3\}$, which is Theorem \ref{thm:two}.

By the definition of $\delta$ in \S3.1(8), we have $x,y \in \phi(M)$ and $d_{\mathbb{R}^N}(x,y) < \delta$ implies  $x+t_1 v_x \neq y+t_2 v_y$ for $|t_i| < \epsilon$ and for any unit normal vectors $v_x, v_y$ at 
$x,y,$ resp. By Lemma \ref{lem:two}, 
 for $d_{\mathbb{R}^N}(x,y) < \delta_{H_t} = \delta_{H_t}(u)$ 
for a fixed normal vector field $u$ of length at most one, we have $x+t u_x \neq y+t u_y$.
(By the remarks above Lemma \ref{lem:two}, we also have $|t|<\epsilon$ here.)  Since $\delta$ does not depend on a choice of vector field $u$, we have $\delta \leq \delta_{H_t}(u).$  This implies
$\delta \leq \delta_H.$  Thus we can conclude that $\phi_t$ is injective   for 
$t^* < \min\{K^{-1}, \delta/3\}$, and the proof of Theorem \ref{thm:two} is complete.

 \begin{rem}\label{rem:one}  {\rm
We review the explicit lower bound for $\delta$.  For 
$L$ defined by (\ref{G}), $\delta^0(q_0,v_0)$ is defined by (\ref{do}).  
 For $P$ defined by (\ref{M}), $\delta^1(q_0,v_0)$ is defined by (\ref{d1}).
For  $S^p$ defined in (\ref{gp}),    
 $\delta^2(q_0,v_0)$ is defined in (\ref{d3}).  Then (\ref{delta est}) defines $\delta(q_0, v_0).$  Finally, (\ref{delta}) defines $\delta.$  
 
 In particular,  lower bounds on  $L,$ $P$, 
 and $S^p$  will give a lower bound on $\delta.$
These constants depend on $q$-derivatives ({\it i.e.}, $M$ coordinate derivatives) of the $\R^N$ coordinates of $\phi$ and of vectors in $\np$ (see {\it e.g.}, (\ref{quick2})).  Since the normal bundle is determined by $M$ and $\phi$,
our estimates are explicit in the sense of Remark \ref{rem:zero}.  }
 \end{rem}

\medskip

\subsection{The Main Theorem}

  Since $M$ is compact and since $\phi_t$ is an injective immersion for $|t| \leq t^*$ by Theorem \ref{thm:two}, 
 by Prop.~\ref{prop:1.1}
  we obtain the main result that $\phi_t$ is an embedding for $t$ less than an explicit 
  $t^*$.

\begin{thm} \label{thm:three} Let $u$ be a normal vector field of length at most one along $\phi(M) \subset \mathbb{R} ^{N}$. Let $t^* = \min\{K^{-1}, \delta/3 \}$, with $K$ defined in \S5.1(7) and $\delta$ estimated in Remark \ref{rem:one}. Then $\phi_t: M \rightarrow \mathbb{R}^N$ given by $m \mapsto \phi(m) + tu_{\phi(m)}$ is an embedding for $|t| \leq t^*$.
\end{thm}

\section{Discussion}
In this paper, we have proposed treating manifold learning by gradient flow techniques that are standard in much of machine learning.  By doing gradient flow in the infinite dimensional space of embeddings of a fixed manifold $M$ into $\R^N$, we avoid parametric and RKHS methods. These methods typically restrict the class of manifolds considered to a finite dimensional space, which speeds up computation time at the cost of perhaps oversimplifying the problem.  In our approach, we give both a theoretical reason to move only in normal directions to the embedded manifold and theoretical lower bounds on the existence for each step of a good discretized version of gradient flow on the space of embeddings.  However, this paper does not discuss computational issues, which must be addressed in future work.  In particular, one has to recompute the estimates for the maximal time $t^*$ of travel after each step. This reflects the theoretical issue that the gradient flow may leave the space of embeddings in finite time. It may be possible to add a penalty term to the objective function that forces the gradient flow to stay in the space of embeddings. This
 new term would involve the bounds we computed on both local quantities like $K$ and global quantities like $\delta$ in \S5.1.

 There are several practical and theoretical issues raised by this approach.  On the practical side, if $M$ flows discretely in $k$ steps to a Riemannian manifold $M_k$ with a thin neck,  as typically happens in mean curvature flow, then in Thm.~\ref{thm:three} $K$ will be very large and $t^* = t^*_k$ will be small at $M_k$.  Thus the 
discretized gradient flow will essentially stop.  It may be reasonable to pick the first $k$ such that $K$ at $M_k$ exceeds a specified threshold. We then backtrack to $M_{k-1}$ (or even further back to some $M_{k-r}$ for some $r>1$) and move to $M'_{k}$ using the gradient at $M_{k-1}$ and new step size $\bar t_k < t^*_k$, {\it e.g.,}  $\bar t_k = (1/2) t^*_k.$ Since the gradient vector field at $M'_{k}$ is different from the gradient vector field at $M_k$, the discretized flow may move $M'_{k}$  to $M'_{k+1}$ with $K$ at $M'_{k+1}$ still below the threshold.  Thus we may be able to extend the flow for an increased number of steps.

There are two theoretical issues that need further examination.  The first is the choice of $M$:  how is this manifold specified?  Based on Riemannian geometry estimates dating to the 1980s, it is reasonable to assume that we want to consider manifolds of a fixed dimension with {\it a priori} a lower bound on volume, an upper bound on diameter, and two-sided bounds on sectional curvature.
Cheeger's finiteness theorem \cite{cheeger} asserts that there are only a finite number of diffeomorphism classes among all such manifolds.  (It would be interesting to determine if the class ${\mathcal G}(d,V,\tau)$ in 
\cite{fefferman2016} has a similar finiteness theorem. We note that the approach of Fefferman {\it et al.} has the strong advantage of not specifying the diffeomorphism type of $M$.) However, while this in theory provides us with a finite list of choices, the proof of the finiteness theorem is nonconstructive.  In practice, in many cases we might as well assume that $M$ is the closed unit ball $B^k$ in $\R^k$. For example, in the famous Swiss roll examples, the data set appears to lie on the image of 
a severely deformed $B^2$.  In contrast, if the training data appears 
 to lie on a deformed torus, $B^2$ is a worse choice for $M$ than the standard torus.

Perhaps even more importantly, it is unclear how to specify the dimension of $M$ in advance.  This has been discussed in the literature: see {\it e.g.} \cite{WangMarron2008} and its references for work done before the last decade,  and \cite{GranataCarnevale2016} for more recent work.  In these works, issues such as the potentially 
fractal/Hausdorff dimension of the data set have been discussed.  From a more geometric mindset,
 we could speculatively start with a $k$-manifold, and hope that in the long run, $M$ would collapse in the sense of Cheeger-Gromov \cite{cheeger-gromov} to a lower dimensional manifold of ``best" dimension.  This would address the issue that the initial choice for $M$ has to be modified as more data is considered.  Even more speculatively, since all Riemannian manifolds are via cut locus arguments homeomorphic to a closed ball with gluings on the boundary, we could start with the $k$-ball $B^k$, add a regularization term, like 
the volume of $\partial B^k = S^{k-1},$ that penalizes the existence of a boundary, and hope that long time flow 
provides both dimension collapse and boundary gluing.  We have no evidence that this will work, but a low dimensional computation is potentially feasible.

\section*{Acknowledgements}
Our thanks to  Carlangelo Liverani for allowing us to use his Quantitative Implicit Function Theorem. We are also grateful to  Qinxun Bai, Andres Larrain-Hubach, Drew Lohn, and the referee for their helpful suggestions.
        
 \appendix     
\section{ The Quantitative Implicit Function Theorem}

This quantitative version of the Implicit Function theorem and its
 proof 
 are from \cite{CL}  (see also \cite[Appendix A]{LC}).\\

For notation, recall that $\Vert A\Vert$ is the sup norm of the absolute values of the entries of a matrix $A$.  For  fixed $(x_0, \lambda_0) \in \R^m \times \R^n$ and fixed $\delta > 0$, set $V_{\delta} = V_{\delta(x_0,\lambda_0)}= \{(x,\lambda) \in \mathbb{R}^{m+n}: |x-x_0| \leq \delta, |\lambda-\lambda_0| \leq \delta \}$.

For $F \in C^1(\mathbb{R}^{m+n}, \mathbb{R}^m)$, let $(x_0, \lambda_0) \in \mathbb{R}^m \times \mathbb{R}^n$ satisfy $F(x_0, \lambda_0) =0 $.

\begin{thm}[Quantitative Implicit Function Theorem]
Assume that the $m\times m$ matrix $\partial_xF(x_0,\lambda_0)$ is invertible and choose $\delta > 
0$ such that 
$$\sup_{(x,\lambda)\in V_{\delta}}||{\rm Id} - [\partial_xF(x_0,\lambda_0)]^{-1}\partial_xF(x,
\lambda)|| \leq 1/2.$$
 Let $B_{\delta} = \sup_{(x,\lambda)\in V_{\delta}}|| \partial_{\lambda}F(x,
\lambda)||$ and $M = ||\partial_xF(x_0,\lambda_0)^{-1}||$. Set $\delta^1 = (2MB_{\delta})^{-1}\delta$, and 
set $\Gamma_{\delta^1} = \{ \lambda \in \mathbb{R}^n : |\lambda - \lambda_0| < \delta^1 \}$,
$V_{\delta,\delta^1} = \{ (x,\lambda) \in \mathbb{R}^{m+n} : | x-x_0| \leq \delta, |\lambda -\lambda_0| \leq \delta^1 \}$.
 
Then there exists $g \in C^1(\Gamma_{\delta^1}, \mathbb{R}^m)$ such that all  solutions of the 
equation $F(x, \lambda) =0$ in the set $V_{\delta,\delta^1}$ 
 are given by $(g(\lambda), \lambda)$.
 In addition, 
 $ \partial_{\lambda}g(\lambda) = - (\partial_x F(g(\lambda), \lambda))^{-1} \partial_{\lambda} F(g(\lambda), \lambda)$. 
\end{thm}

\begin{proof} 

 Take $\lambda\in \Gamma_{\delta^1} = \{ |\lambda - \lambda_0| < \delta^1\}$. 
 Consider $U_{\delta} = \{x \in \mathbb{R}^m : |x-x_0| \leq \delta \}$ and $\Omega_\lambda : U_{\delta} \rightarrow \mathbb{R}^m$ defined by
$$
 \Omega_{\lambda}(x) = x- \partial_xF(x_0,\lambda_0)^{-1}F(x, \lambda).
 $$
For $x \in U_\delta, F(x, \lambda) =0$ is equivalent to $x = \Omega_{\lambda} (x)$. 
We have 
$$
  |\Omega_{\lambda}(x_0) - \Omega_{\lambda_0}(x_0)| \leq M | F(x_0,\lambda) - F (x_0, \lambda_0)| \leq MB_{\delta}\delta^1. 
$$
In addition, $|\partial_x \Omega_{\lambda}| =| {\rm Id}  - \partial_xF(x_0, \lambda_0)^{-1}
\partial_xF(x, \lambda)| \leq 1/2$, so  
$| \Omega_{\lambda}(x) -\Omega_\lambda(x_0)| \leq \frac{1}{2}|x-x_0|.$
Thus
\begin{align*}
|\Omega_{\lambda}(x) -x_0| &\leq | \Omega_{\lambda}(x) -\Omega_{\lambda}(x_0)| + 
|\Omega_\lambda(x_0) - x_0|\\
&\leq 
\frac{1}{2} |x-x_0| + MB_{\delta}\delta^1 \leq \delta. 
 \end{align*}
 Thus $\Omega_\lambda$ is a contraction on $U_\delta$, and $\Omega_\lambda(x) = x$ has a unique solution $x = g(\lambda)$ by the Contraction Fixed Point Theorem.
We have therefore obtained a function $g: \Gamma_{\delta^1} 
\rightarrow U_\delta$ such that $F(g(\lambda),\lambda) =0$. All solutions in $V_{\delta,\delta^1}$ are of this form: if $F(x_1,\lambda_1) = 0$, then
$$ |x_1-g(\lambda_1)| = |\Omega_{\lambda_1}(x_1) - \Omega_{\lambda_1}(g(\lambda_1))|
\leq \frac{1}{2} |x_1 - g(\lambda_1)|,$$
so $x_1 = g(\lambda_1).$

For the final statement in the Theorem, let $\lambda, \lambda' \in \Gamma_{\delta^1}$. As above, we have
$$
|g(\lambda) - g(\lambda') | \leq \frac{1}{2} |g(\lambda) - g(\lambda') | + M B_{\delta} |\lambda - \lambda '| $$
This yields the Lipschitz continuity of $g$. To obtain  differentiability,  
by Taylor's theorem for $F\in C^1$ and the  Lipschitz continuity of $g$, we obtain,  for $h \in \mathbb{R}^n$, 
\begin{align*}\MoveEqLeft{0 = \lim_{|h|\to 0} |h|^{-1}|F(g(\lambda + h),\lambda+h) - F(g(\lambda), \lambda)| }\\
&= \lim_{|h|\to 0} |h|^{-1} | F(g(\lambda + h),\lambda+h)- F(g(\lambda), \lambda+h)
+ F(g(\lambda), \lambda+h) - F(g(\lambda),\lambda)| \\
&= \lim_{|h|\to 0} |h|^{-1} |\partial_xF(g(\lambda),\lambda+h)(g(\lambda+h) - g(\lambda))
+  \partial_\lambda F(g(\lambda),\lambda)(\lambda+h - \lambda)|\\
&=  \partial_xF(g(\lambda),\lambda) \lim_{h\to 0}|h|^{-1} |g(\lambda+h) - g(\lambda) 
+ (\partial_xF(g(\lambda),\lambda) )^{-1} |\partial_\lambda F(g(\lambda),\lambda)|.
\end{align*}
Since $ \partial_xF(g(\lambda),\lambda)\neq 0,$  we get
$ \partial_{\lambda}g(\lambda) = - (\partial_x F(g(\lambda), \lambda))^{-1} \partial_{\lambda} F(g(\lambda), \lambda)$. 
\end{proof}
  
 \bibliographystyle{amsplain}
\bibliography{Gold_Rosenberg_GradientFlow}  

\providecommand{\bysame}{\leavevmode\hbox to3em{\hrulefill}\thinspace}
\providecommand{\MR}{\relax\ifhmode\unskip\space\fi MR }
\providecommand{\MRhref}[2]{%
  \href{http://www.ams.org/mathscinet-getitem?mr=#1}{#2}
}
\providecommand{\href}[2]{#2}
\begin{thebibliography}{10}

\bibitem{AGS}
Luigi Ambrosia, Nicola Gigli, and Giuseppe Savar\'e, \emph{Gradient {F}lows in
  {M}etric {S}paces and in the {S}pace of {P}robability {M}easures},
  Birkh\"auser, Basil, 2008.

\bibitem{AD}
Mich\`ele Audin and Mihai Damian, \emph{Morse theory and {F}loer homology},
  Universitext, Springer, London; EDP Sciences, Les Ulis, 2014.

\bibitem{BRSW}
Qinxun Bai, Steven Rosenberg, Zheng Wu, and Stan Sclaroff, \emph{A differential
  geometric approach to classification}, Proceedings of The 33rd International
  Conference on Machine Learning \textbf{48} (2016).

\bibitem{BR}
Qinxun Bai, Steven Rosenberg, and Wei Xu, \emph{A geometric understanding of
  natural gradient}, \url{https://arxiv.org/abs/2202.06232}.

\bibitem{Belkin}
Mihail Belkin and Partha Niyogi, \emph{Laplacian eigenmaps for dimensionality
  reduction and data representation}, Neural Computation \textbf{15} (2003),
  1373--1396.

\bibitem{BNS}
Mikhail Belkin, Partha Niyogi, and Vikas Sindhwani, \emph{Manifold
  regularization: A geometric framework for learning from labeled and unlabeled
  examples}, Journal of Machine Learning Research \textbf{7} (2006),
  2399--2434.

\bibitem{Bergmann}
Ronny Bergmann~{\it et al.}, \emph{Discrete total variation of the normal
  vector field as shape prior with applications in geometric inverse problems},
  Inverse Problems \textbf{36} (2020).

\bibitem{cheeger}
Jeff Cheeger, \emph{Finiteness theorems for {R}iemannian manifolds}, Amer. J.
  Math. \textbf{92} (1970), 61--74.

\bibitem{cheeger-gromov}
Jeff Cheeger and Mikhael Gromov, \emph{Collapsing {R}iemannian manifolds while
  keeping their curvature bounded. {I}}, J. Differential Geom. \textbf{23}
  (1986), no.~3, 309--346.

\bibitem{LC}
Luigi Chierchia, \emph{Kolomogorov-{A}rnold-{M}oser ({K}{A}{M}) theory},
  Mathematics of {C}omplexity and {D}ynamical {S}ystems. {V}ols. 1--3,
  {S}pringer, {N}ew {Y}ork (2012), 810--836.

\bibitem{Cooper}
Yaim Cooper, \emph{Discrete gradient descent differs qualitatively from
  gradient flow}, arXiv:1808.04839 (2018).

\bibitem{CSK}
Antonio Criminisi, Jamie Shotton, and Ender Konukoglu, \emph{Decision forests:
  A unified framework for classification, regression, density estimation,
  manifold learning and semi-supervised learning}, Foundations and Trends in
  Computer Graphics and Vision \textbf{7} (2012), 81--227.

\bibitem{DonohoGrimes2003}
David Donoho and Carrie Grimes, \emph{Hessian eigenmaps: Locally linear
  embedding techniques for high-dimensional data}, Proceedings of the National
  Academy of Sciences \textbf{100} (2003), no.~10, 5591--5596.

\bibitem{DR}
Mark Droske and Martin Rumpf, \emph{A variational approach to nonrigid
  morphological image registration}, SIAM Journal on Applied Mathematics
  \textbf{2} (2004), 668--687.

\bibitem{Eells}
James Eells, Jr., \emph{A setting for global analysis}, Bull. Amer. Math. Soc.
  \textbf{72} (1966), 751--807.

\bibitem{feffermanReconstructionII}
Charles Fefferman, Sergei Ivanov, Yaroslav Kurylev, Matti Lassas, Jinpeng Lu,
  and Hariharan Narayanan, \emph{Reconstruction and interpolation of manifolds.
  {II}: {I}nverse problems for {R}iemannian manifolds with partial distance
  data}, \url{https://arxiv.org/abs/2111.14528}.

\bibitem{feffermanReconstructionI}
Charles Fefferman, Sergei Ivanov, Yaroslav Kurylev, Matti Lassas, and Hariharan
  Narayanan, \emph{Reconstruction and interpolation of manifolds. {I}: {T}he
  geometric {W}hitney problem}, Found. Comput. Math. \textbf{20} (2020), no.~5,
  1035--1133.

\bibitem{fefferman2019}
Charles Fefferman, Sergei Ivanov, Matti Lassas, and Hariharan Narayanan,
  \emph{Fitting a manifold of large reach to noisy data},
  \url{https://arxiv.org/abs/1910.05084}.

\bibitem{fefferman2020}
\bysame, \emph{Reconstruction of a {R}iemannian manifold from noisy intrinsic
  distances}, SIAM J. Math. Data Sci. \textbf{2} (2020), no.~3, 770--808.

\bibitem{fefferman2016}
Charles Fefferman, Sanjoy Mitter, and Hariharan Narayanan, \emph{Testing the
  manifold hypothesis}, J. Amer. Math. Soc. \textbf{29} (2016), no.~4,
  983--1049.

\bibitem{Ger}
Claus Gerhardt, \emph{Evolutionary surfaces of prescribed mean curvature},
  Journal of Differential Equations \textbf{36} (1980), 139--172.

\bibitem{GranataCarnevale2016}
Daniele Granata and Vincenzo Carnevale, \emph{Accurate estimation of the
  intrinsic dimension using graph distances: Unraveling the geometric
  complexity of datasets}, Sci. Rep. \textbf{6} (2016),
  \url{https://www.nature.com/articles/srep31377}.

\bibitem{GFDH}
Guodong Guo, Yun Fu, Charles~R. Dyer, and Thomas~S. Huang, \emph{Image-based
  human age estimation by manifold learning and locally adjusted robust
  regression}, IEEE Transactions on Image Processing \textbf{17} (2008),
  1178--1188.

\bibitem{Ham2}
Richard~S. Hamilton, \emph{Harnack estimate for the mean curvature flow},
  Journal of Differential Geometry \textbf{41} (1995), 215--226.

\bibitem{HS}
Gerhard Huisken and Carlo Sinestrari, \emph{Mean curvature flow singularities
  for mean convex surfaces}, Calculus of Variations and Partial Differential
  Equations \textbf{8} (1999), 1--14.

\bibitem{Lee}
John~M. Lee, \emph{Introduction to {S}mooth {M}anifolds}, Graduate Texts in
  Mathematics, vol. 218, Springer, New York, 2013.

\bibitem{Lin2015}
Tong Lin, Hanlin Xue, Ling Wang, Bo~Huang, and Hongbin Zha, \emph{Supervised
  learning via {E}uler's elastica models}, Journal of Machine Learning Research
  \textbf{16} (2015), 3637--3686.

\bibitem{CL}
Calangelo Liverani, \emph{Implicit function theorem (a quantitative version)},
  \url{https://www.mat.uniroma2.it/~liverani/Calcolo1-2016/implicit.pdf}.

\bibitem{MF}
Yunqian Ma and Yun Fu (eds.), \emph{Manifold {L}earning and {A}pplications},
  CRC Press, Boca Raton, 2011.

\bibitem{May}
Uwe~F. Mayer, \emph{Gradient flows on nonpositively curved metric spaces and
  harmonic maps}, Communications in Analysis and Geometry \textbf{6} (1998),
  no.~2, 199--253.

\bibitem{Mil}
John Milnor, \emph{Morse {T}heory}, Princeton University Press, Princeton, NJ,
  1969.

\bibitem{Morse}
Marston Morse, \emph{The foundations of the calculus of variations in m-space.
  {P}art {I}}, Trans. Amer. Math. Soc. \textbf{31} (1929), 379--404.

\bibitem{Mumford1989}
David Mumford and Jayant Shah, \emph{Optimal approximations by piecewise smooth
  functions and associated variational problems}, Communications in Pure and
  Applied mathematics \textbf{42} (1989), no.~5, 577--685.

\bibitem{Omori}
Hideki Omori, \emph{Infinite-dimensional {L}ie groups}, Translations of
  Mathematical Monographs, vol. 158, American Mathematical Society, Providence,
  RI, 1997.

\bibitem{OS}
Stanley Osher and James.~A Sethian, \emph{Fronts propogating with curvature
  dependant speed: Algorithms based on {H}amilton-{J}acobi formulations},
  Journal of Computational Physics \textbf{79} (1988), 12--49.

\bibitem{RS2000}
Sam Roweis and Lawrence Saul, \emph{Nonlinear dimensionality reduction by
  locally linear embedding}, Science \textbf{290} (2000), no.~5500, 2323--2326.

\bibitem{RT}
Melanie Rupflin and Peter~M. Topping, \emph{Flowing maps to minimal surfaces},
  American Journal of Mathematics \textbf{138} (2016), no.~4, 1095--1115.

\bibitem{Sethian1999}
James~A. Sethian, \emph{Level {S}et {M}ethods and {F}ast {M}arching {M}ethods:
  {E}volving {I}nterfaces in {C}omputational {G}eometry, {F}luid {M}echanics,
  {C}omputer {V}ision, and {M}aterials {S}cience}, vol.~3, Cambridge
  {U}niversity {P}ress, 1999.

\bibitem{Smola2001}
Alexander Smola, Sebastian Mika, Bernhard Sch\"olkopf, and Robert Williamson,
  \emph{Regularized principal manifolds}, JMLR \textbf{1} (2001), 179--209.

\bibitem{Tennenbaum}
Joshua Tenenbaum, Vin de~Silva, and John Langford, \emph{A global geometric
  framework for nonlinear dimensionality reduction}, Science \textbf{290}
  (2000), 2319--2323.

\bibitem{VW}
Kush Varshney and Alan Willsky, \emph{Classification using geometric level
  sets}, Journal of Machine Learning Research \textbf{11} (2010), 491--516.

\bibitem{WangMarron2008}
Xiaohui Wang and J.~S. Marron, \emph{A scale-based approach to finding
  effective dimensionality in manifold learning}, Electron. J. Stat. \textbf{2}
  (2008), 127--148.

\bibitem{Xiao}
Ling Xiao, \emph{Gradient estimates and lower bound for the blow-up time of
  star shaped mean curvature flow}, \url{https://arxiv.org/pdf/1311.3721.pdf}.

\bibitem{Ye}
Ye~Yuan and Chuanjiang He, \emph{Variational level set methods for image
  segmentation based on both {L}2 and {S}obolev gradients}, Nonlinear Analysis:
  Real World Applications \textbf{13} (2012), 959--966.

\end{thebibliography}

\end{document}